\documentclass[12pt, reqno]{amsart}
\allowdisplaybreaks[1]
\usepackage{amsmath}
\usepackage{amssymb}
\usepackage{amsfonts}
\usepackage{verbatim}
\usepackage[usenames]{color}
\usepackage{hyperref}
\makeindex

 \newtheorem{theorem}{Theorem}[section]

\theoremstyle{definition}

\theoremstyle{remark}

\newtheorem{fact*}{Fact}

\newcommand{\til}{\raise.17ex\hbox{$\scriptstyle\mathtt{\sim}$}}

\newcommand\beq{\begin{equation}}

\newcommand\eeq{\end{equation}}

\newcommand{\bbm}{\left[ \begin{smallmatrix}}
\newcommand{\ebm}{\end{smallmatrix} \right]}
\newcommand{\bpm}{\left( \begin{smallmatrix}}
\newcommand{\epm}{\end{smallmatrix} \right)}
\numberwithin{equation}{section}

\newlength{\Mheight}
\newlength{\cwidth}

\newcommand{\dfn}[1]{{\bf #1}\index{#1}}

\title[Rational noncommutative Positivstellensatz\"e]{Positivstellensatz\"e for noncommutative rational expressions}
\author{
J. E. Pascoe 
}
\address{Department of Mathematics, Washington University in St. Louis, 1 Brookings Drive, Campus Box 1146, St. Louis, MO 63130, USA.}
\email{pascoej@math.wustl.edu}
\thanks{Research supported by NSF Mathematical Science Postdoctoral Research Fellowship DMS 1606260.}
\date{\today}

\keywords{Noncommutative rational function, positive rational function, Hilbert’s 17th problem, noncommutative Positivstellensatz}
 \subjclass[2010]{Primary 13J30, 16K40, 47L07; Secondary 15A22, 26C15, 47A63}

\begin{document}

\begin{abstract}
We derive some Positivstellensatz\"e for noncommutative rational expressions from the Positivstellensatz\"e for
noncommutative polynomials. Specifically, we show that if a noncommutative rational expression is positive on a polynomially convex set, then there is an algebraic certificate witnessing that fact. As in the case of noncommutative polynomials, our results are nicer when we additionally assume positivity on a convex set-- that is, we obtain a so-called ``perfect Positivstellensatz" on convex sets.
\end{abstract}
\maketitle

\section{Introduction}

We consider the positivity of noncommutative rational functions on polynomially convex
sets. The theory on positive noncommutative polynomials has been well-studied
\cite{heltonPositive, helmc04, helkm12}, essentially inspired by the operator theoretic methods from the theory of positive (commutative) polynomials
on polynomially convex sets originating in the work  \cite{schmu91, put93}. We note that going from the polynomial to the rational case is less clear than in the noncommutative case because we cannot ``clear denominators," as it were.

A \dfn{noncommutative polynomial} (over $\mathbb{C}$) in $d$-variables is an element of the free associative algebra over $\mathbb{C}$ in the noncommuting letters $x_1, \ldots, x_d.$
For example $1000x_1x_2x_1 - x_2^2$ and $x_1^2 + x_1x_2$ are noncommutative  polynomials in two variables.
A \dfn{matricial noncommutative  polynomial} is a matrix with noncommutative  polynomial entries. For example,
	$$\bbm
		7i & 1000x_1x_2x_1 - x_2^2 \\
		x_1^2 + x_1x_2 & 0
	\ebm$$
is a matricial noncommutative  polynomial. We define an involution $*$ on matricial noncommutative  polynomials to be the involution which treats each $x_i$ as a self-adjoint variable. For example,
	$$\bbm
		7i & 1000x_1x_2x_1 - x_2^2 \\
		x_1^2 + x_1x_2 & 0
	\ebm^* = \bbm
		-7i & x_1^2 + x_2x_1 \\
		1000x_1x_2x_1 - x_2^2 & 0
	\ebm.$$

We say a  collection $\mathcal{P}$ of square matricial noncommutative  polynomials is \dfn{Archimedian} if
$\mathcal{P}$ contains elements of the form $C_i - x_i^2$ for some real numbers $C_i$
and each element of $\mathcal{P}$ is self-adjoint.

Let $\mathcal{H}$ be the infinite dimensional separable Hilbert space.
For a self-adjoint operator $T$, we say $T \geq 0$ if $T$ is positive semidefinite, we say $T > 0$ if $T$ is strictly positive definite in the sense that the spectrum of $T$ is contained in $(0,\infty)$.
We define $$\mathcal{D}_\mathcal{P} = \{ X \in B(\mathcal{H})^d | p(X) \geq 0, \forall p \in \mathcal{P}, X_i = X_i^*\}.$$

Previously, Helton and McCullough showed the following Positivstellensatz for matricial
noncommutative  polynomials.
\begin{theorem}[Helton, McCullough \cite{helmc04}]\label{HMC}
Let $\mathcal{P}$ be an Archimedian collection of matricial noncommutative  polynomials.
Let $q$ be a square matricial noncommutative  polynomial.
If $q > 0$ on $\mathcal{D}_\mathcal{P},$
then 
	$$q = \sum_{\text{finite}} s_i^*s_i + \sum_{\text{finite}} r_j^*p_jr_j$$
where $s_i, r_j$ are all matricial noncommutative  polynomials and $p_j \in \mathcal{P}.$
\end{theorem}

\section{The rational positivstellensatz}

A \dfn{noncommutative  rational expression} is a syntactically correct expression involving
	$+, (, ), ^{-1}$ the letters $x_1, \ldots, x_d$ and scalar numbers.
We say two nondegenerate expressions are \dfn{equivalent} if they agree on the intersection of their domains.
(Nondegeneracy means that the expression is defined for at least one input, or equivalently that the domain is a dense set with interior. That is, examples such as $0^{-1}$ are disallowed.)
Examples of noncommutative  rational expressions include
		$$ 1, x_1x_1^{-1}, 1+ x_2(8x_1^3x_2x_1+8)^{-1}.$$
We note that the first two are equivalent.

A \dfn{matricial noncommutative  rational expression} is a matrix with noncommutative  rational expression entries.


We show the following theorem.
\begin{theorem}\label{SoSthm}
Let $\mathcal{P}$ be an Archimedian collection of noncommutative  polynomials.
Let $q$ be a square matricial noncommutative  rational expression defined on all of $\mathcal{D}_{\mathcal{P}}$.
If the noncommutative rational expression $q > 0$ on $\mathcal{D}_\mathcal{P},$
then
	\beq \label{SoS}
	q \equiv \sum_{\text{finite}} s_i^*s_i + \sum_{\text{finite}} r_j^*p_jr_j
	\eeq
where $s_i, r_j$ are all matricial noncommutative  rational expressions defined on $\mathcal{D}_{\mathcal{P}}$ and $p_j \in \mathcal{P}.$
\end{theorem}
\begin{proof}
We let $g_j(x)$ be such that the term $g_j(x)^{-1}$ occurs in $q.$
The proof will go by strong induction on the number of such terms.
Define $$\mathcal{O} = \mathcal{P} \cup \{ \pm[1 - u_jg_j(x)]^*[1 - u_jg_j(x)],
 \pm[1 - g_j(x)u_j]^*[1 - g_j(x)u_j]\}
\cup \{D_j - u_j^*u_j\}$$
where $D_j$ are positive real scalars chosen to be large enough so that $D_j - [g_j(x)^{-1}]^* g_j(x)^{-1}$
is positive on $\mathcal{D}_{\mathcal{P}}.$

We now define a self-adjoint noncommutative  polynomial $\hat{q}(x,u)$ so that
$\hat{q}(x,g) = q(x).$
Now $\hat{q}$ is a noncommutative  polynomial in terms of $x_i$ and $u_j.$
Moreover, in terms of the $x_i$ and $u_j,$ we see that $q(x,u)$ is positive on
$\mathcal{D}_{\mathcal{O}},$
so by Theorem \ref{HMC},
	$$\hat{q} = \sum s_i^*s_i + \sum r_j^*o_jr_j$$
for some $o_j \in \mathcal{O}.$
We now analyze each term of the form $t_j = r_j^*o_jr_j.$
We need to show that $t_j(x,g)$ is of the form \eqref{SoS}.
If $o_j \in \mathcal{P}$, we are fine.
If $o_j = \pm[1 - u_jg_j(x)]^*[1 - u_jg_j(x)],$ we are also fine, since $t_j(x,g) = 0,$ and similarly for
the reversed case.
If $o_j = D_j - [u_j]^* u_j$ we note that
	$$o_j(x,g) = D_j - [g_j(x)^{-1}]^*g_j(x)^{-1} =
[g_j(x)^{-1}]^*[D_jg_j(x)^*g_j(x) - 1]g_j(x)^{-1},$$
and since $D_jg_j(x)^*g_j(x) - 1>0$ on $\mathcal{D}_{\mathcal{P}},$
by induction it is of the form \eqref{SoS}, so we are done.
\end{proof}
We note that the same proof can be adapted for the hereditary case in \cite{helmc04}. Moreover, we note that this implies the Agler model theory for rational functions on polynomially convex sets established variously in \cite{Ball2006, agmc_gh}.


\section{The convex perfect rational positivstellensatz}
It is important to note that in Theorem \ref{HMC} and Theorem \ref{SoSthm}, the complexity of the sum of squares representation is unbounded and we needed strict inequality. Specifically, in \eqref{SoS},
the number of terms in each sum and the degree of each $s_i$ and $r_j$ are not bounded in the statement of the theorem.
However, Helton, Klep and McCullough \cite{helkm12} showed that
bounds do exist when we additionally assume that $\mathcal{D}_{\mathcal{P}}$ is convex and contains $0$ and moreover that $\mathcal{P}$  consists of a single monic linear pencil, $L$, a self-adjoint linear matrix polynomial such that $L(0)$ is the identity.
We note that for any finite set $\mathcal{P}$ of noncommutative polynomials
such that $\mathcal{D}_{\mathcal{P}}$  is convex and contains $0$, there exists such an $L$ \cite{heltmc12}.

Our goal is to prove the following:
\begin{theorem}
	Let $L$ be a monic linear pencil.
	Suppose  $\mathcal{D}_{\{L\}}$ is convex.
Let $r$ be a square matricial noncommutative  rational expression defined on all of $\mathcal{D}_{\{L\}}$.
The noncommutative rational expression $r \geq 0$ on all of  $\mathcal{D}_{\{L\}}$ if and only if
 \beq
	r \equiv \sum_{\text{finite}} s_i^*s_i + \sum_{\text{finite}} r_j^*L r_j
	\eeq
	where $s_i, r_j$ are all matricial noncommutative  rational expressions defined on all of $\mathcal{D}_{\{L\}}.$
\end{theorem}

\begin{proof}
Given an expression $r(x),$ we consider the expression $\tilde{r}(x,u)$ where each $g_j(x)^{-1}$ occurring in $r$ has been replaced by $u_j$ as in the proof of Theorem \ref{SoSthm}.

First we consider the minimal set $\mathcal{C}_r$ of rational expressions such that:
\begin{enumerate}
	\item $ab \in \mathcal{C}_r \Rightarrow b \in \mathcal{C}_r,$
	\item $(a+b)c \in \mathcal{C}_r \Rightarrow ac \in \mathcal{C}_r, 
	bc \in \mathcal{C}_r,$
	\item $a+b \in \mathcal{C} \Rightarrow a \in \mathcal{C}_r, 
	b \in \mathcal{C}_r,$
	\item $a^{-1}b \in \mathcal{C}_r \Rightarrow aa^{-1}b\in \mathcal{C}_r.$
\end{enumerate}
From $\mathcal{C}_r$,
form a set
 $\tilde{\mathcal{C}}_r$
by replacing each occurence of $g_j(x)^{-1}$
in elements of $\mathcal{C}_r$ with a new symbol $u_j.$
We define the set of
 $\mathcal{M}_r$ to be
	$$\mathcal{M}_r = \{ g_j(x)u_jb - b | g_j(x)u_jb \in \tilde{\mathcal{C}}_r\}.$$
	Define
			$$\mathcal{Z}_r = \{(X, U, v)| m(X, U)v = 0, m\in \mathcal{M}_r, L(X)\geq 0\}.$$
	We note that for $(X, U, v) \in \mathcal{Z}_r$
	and $\tilde{a}(x,u) \in \tilde{\mathcal{C}_r},$
	one can show
	we have that $\tilde{a}(X,U)v = \tilde{a}(x,g(X)^{-1})v$
	via a recursive argument.
	We see that $\tilde{r}(x, u)$
	satisfies
		 $$\langle r(X)v , v \rangle = \langle \tilde{r}(X, U)v , v \rangle \geq 0,$$
	on $\mathcal{Z}_r$
	since $\tilde{r}(X,U)v = r(X)v$ on $\mathcal{Z}_r$ by construction.
	Now, we apply the Helton-Klep-Nelson convex Positivstellensatz\cite[Theorem 1.9]{HKN}, where the variety is given by $\mathcal{Z}_r$ and the convex set is $\{(X,U) | L(X)\geq 0\}$, to get that:
	$$\tilde{r}(x, u) = \sum \tilde{s}_i^*\tilde{s}_i + \sum \tilde{r}_j^* L \tilde{r}_j + \sum \iota_k^* m_k + m_k^* \iota_k$$
	where each $\iota_k$ is in the real radical of the ideal 
	generated by the elements of $\mathcal{M}_r.$ That is,
	each $\iota_k(X,U)v$ vanishes on $\mathcal{Z}_r.$
	So, substituting $g_j(x)^{-1}$ for $u_j$ we get that
	$$r(x) \equiv \sum s_i^*s_i + \sum r_j^* L r_j.$$
\end{proof}

We note that we could have proved a bit more: that on the variety $\mathcal{Z}_r$ that $\tilde{r}$ is positive and given by a sum of squares. This would essentially correspond to the so-called
Moore-Penrose evaluation in \cite{KPV}. Moreover, we note that the main result on positive rational functions, the noncommutative analogue of Artin's solution to Hilbert's seventeenth problem, that regular positive rational expressions are sums of squares\cite{KPV}, follows from our present theorem by taking an empty monic linear pencil, in fact, we obtain a slightly better matricial version of that result. Moreover, one has size bounds inherited from the Helton-Klep-Nelson convex Positivstellensatz \cite{HKN}, that is, checking that a noncommutative rational expression is effective using the algorithms given in\cite{HKN}.

\bibliography{references}
\bibliographystyle{plain}

\end{document}